\documentclass[final,authoryear,11pt,a4paper]{elsarticle}
\usepackage[margin=1in]{geometry}
\usepackage{enumitem,graphicx,amsfonts,amsmath,amssymb,amsthm,doi,tikz,pgfplots,booktabs,float,url,lineno,colortbl,tabulary,csquotes,hyperref,cleveref}
\pgfplotsset{compat=newest}
\graphicspath{{../Matlab.files/}{./}}
\usetikzlibrary{matrix}
\usepackage[labelsep=period]{caption}
\tikzset{font=\footnotesize} 
\journal{Annals of Operations Research}

\newcolumntype{K}{>{\hspace*{-0.6\tabcolsep}}l}
\newcolumntype{M}{l<{\hspace*{-0.6\tabcolsep}}}
\newcolumntype{C}{>{\hspace*{-0.6\tabcolsep}}c}
\crefname{subsection}{section}{subsections}
\newtheorem{theorem}{Theorem}[section]
\newtheorem{lemma}[theorem]{Lemma}

\theoremstyle{definition}
\newtheorem{remark}[theorem]{Remark}
\newtheorem{definition}[theorem]{Definition}
\newtheorem{example}[theorem]{Example}

\newcommand{\ex}{\mathbb{E}}
\newcommand{\F}{\bar{F}}

\newcommand{\m}{\mathrm{m}}
\newcommand{\h}{\mathrm{h}}
\newcommand{\e}{\mathrm{\ell}}
\newcommand{\qt}{\enquote}
\renewcommand{\b}{\bfseries}
\renewcommand{\d}{\mathrm{d}}
\newcommand{\g}{\mathrm{g}}
\newcommand{\us}{\Pi^*_s\(r\)}

\newcommand{\du}{\mathrm{d}u}
\renewcommand{\(}{\left(}
\renewcommand{\)}{\right)}

\usepackage{upgreek}
\newcommand{\G}{\bar{G}}
\renewcommand{\(}{\left(}
\renewcommand{\)}{\right)}
\newcommand{\mrl}{\preceq_{\text{mrd}}}
\newcommand{\hr}{\preceq_{\text{hr}}}
\newcommand{\st}{\preceq_{\text{st}}}

\newcommand{\cx}{\preceq_{\text{cx}}}
\newcommand{\disp}{\preceq_{\text{disp}}}
\newcommand{\ew}{\preceq_{\text{ew}}}
\newcommand{\var}{\operatorname{Var}}
\newcommand{\el}{\preceq_{\text{el}}}

\newcommand{\A}{{\text{Agg}}}
\usepackage{color,colortbl}

\definecolor{ppcolor}{rgb}{0.212,0.506,0.584}
\definecolor{pcolor}{rgb}{0.255,0.596,0.686}
\definecolor{rcolor}{rgb}{0.294,0.675,0.776}
\definecolor{ecolor}{rgb}{0.569,0.765,0.835}
\definecolor{scolor}{rgb}{0.733,0.843,0.890}
\definecolor{ocolor}{rgb}{0.783,0.893,0.940}
\definecolor{oocolor}{rgb}{0.883,0.893,0.940}

\newcommand{\clrone}{\rowcolor{ocolor}}
\newcommand{\clrtwo}{\rowcolor{ecolor}}

\newcommand{\bottomline}{\noalign{\global\arrayrulewidth=0.3mm}\arrayrulecolor{pcolor}\hline}

\definecolor{LightCyan}{rgb}{0.8,1,1}
\definecolor{LightCyan2}{rgb}{0.1,1,1}
\newcommand{\specialcell}[2][l]{%
  \begin{tabular}[#1]{@{}l@{}}#2\end{tabular}}

\usepackage[draft]{changes}

\begin{document}
\begin{frontmatter}

\title{Monopoly Pricing in Vertical Markets with Demand Uncertainty}

\author[kaq]{Stefanos Leonardos\texorpdfstring{\corref{cor}}{l}}\ead{stefanos\_leonardos@sutd.edu.sg}
\author[kap]{Costis Melolidakis}\ead{cmelol@math.uoa.gr}
\author[kat]{Constandina Koki}\ead{kokiconst@aueb.gr}
\address[kaq]{Singapore University of Technology and Design, 8 Somapah Rd, 487372 Singapore}
\address[kap]{National and Kapodistrian University of Athens, Panepistimioupolis, 15784 Athens, Greece}
\address[kat]{Athens University of Economics and Business, 76 Patission Street, 10434 Athens, Greece}
\cortext[cor]{Corresponding author}

\begin{abstract}
\textbf{Motivation:} Pricing decisions are often made when market information is still poor. While modern pricing analytics aid firms to infer the distribution of the stochastic demand that they are facing, data-driven price optimization methods are often impractical or incomplete if not coupled with testable theoretical predictions. In turn, existing theoretical models often reason about the response of optimal prices to changing market characteristics without exploiting all available information about the demand distribution. \textbf{Academic/practical relevance:} Our aim is to develop a theory for the optimization and systematic comparison of prices between different instances of the same market under various forms of knowledge about the corresponding demand distributions. \textbf{Methodology:} We revisit the classic problem of monopoly pricing under demand uncertainty in a vertical market with an upstream supplier and multiple forms of downstream competition between arbitrary symmetric retailers. In all cases, demand uncertainty falls to the supplier who acts first and sets a uniform price before the retailers observe the realized demand and place their orders. \textbf{Results:} Our main methodological contribution is that we express the price elasticity of expected demand in terms of the \emph{mean residual demand} (MRD) function of the demand distribution. This leads to a closed form characterization of the points of unitary elasticity that maximize the supplier's profits and the derivation of a mild unimodality condition for the supplier's objective function that generalizes the widely used \emph{increasing generalized failure rate} (IGFR) condition. A direct implication is that optimal prices between different markets can be ordered if the markets can be \emph{stochastically ordered} according to their MRD functions or equivalently, their elasticities. Using the above, we develop a systematic framework to compare optimal prices between different market instances via the rich theory of \emph{stochastic orders}. This leads to comparative statics that challenge previously established economic insights about the effects of market size, demand transformations and demand variability on monopolistic prices. \textbf{Managerial implications:} Our findings complement data-driven decisions regarding price optimization and provide a systematic framework useful for making theoretical predictions in advance of market movements.
\end{abstract}

\begin{keyword} Monopoly Pricing \sep Revenue Maximization \sep Demand Uncertainty \sep Pricing Analytics \sep Comparative Statics \sep Stochastic Orders \sep Unimodality
\MSC[2010] 91A10 \sep 91A65 \sep 91B54
\end{keyword}

\end{frontmatter}

\section{Introduction}\label{sec:intro}

Making optimal pricing decisions is a crucial driver for firms' profitability and a well-studied problem in the existing theoretical literature. However, modern markets pose novel opportunities and challenges for firms' pricing decisions. On the one hand, the sheer amount of historical price/demand data and the wide range of pricing analytics methods allow firms to make more informed decisions. One the other hand, the inherent volatility of contemporary economies frequently renders such data-driven methods impractical. Sellers, often launch new or differentiated products for which demand is unknown or introduce existing products to uncharted emerging markets \citep{Co16}. In other cases, firms act as wholesalers in foreign markets for which they have asymmetrically less information than local retailers or sell their products over competitive digital platforms to highly diversified clienteles \citep{Che18}. More generally, firms often need to test the outcome of price changes in advance of anticipated market movements or to constantly adjust their prices in periods of turbulent market conditions.\par
Common in all these cases is that firms have to make important pricing decisions when market information is still poor \citep{Li05}. While uncertainties can be mitigated via marketing strategies or contracting schemes between the members of the supply chain, after all efforts, some uncertainty persists and the final point of interaction between wholesalers and retailers or more generally, between sellers and buyers is the selling price \citep{Li17,Ber19}. \par
Whenever possible, pricing analytics on historic price/demand data aid firms to build up knowledge about the (probability) distribution of the uncertain demand. The main challenge lies in leveraging this information to set and adjust prices optimally. However, data-driven approaches that extrapolate past trends to make forward-looking pricing decisions may lead to suboptimal decisions if not coupled with or benchmarked against testable theoretical predictions. In turn, existing theoretical tools to optimize and compare prices across different market instances (instances of the same market that correspond to different demand distributions) are still under development \citep{Xu10} or make partial use of the available information, e.g., rely on summary statistics \citep{La01}. Moreover, from a managerial perspective, such methods often provide optimality conditions that are not easy to assess in practice or which do not provide intuitive and economic interpretable results \citep{Be07}.

\paragraph{Model}
Motivated by the above, we revisit the classic problem of monopoly pricing with demand uncertainty under the informational assumption that the firm knows the probability distribution of the uncertain demand. This distribution may reflect the seller's informed belief or estimations aggregated from historical data. Our purpose is to link the properties of the demand distribution to economic interpretable conditions and to develop a systematic theoretical framework in which the firm can optimally set and adjust its prices according to changing market characteristics.\par
To account for the large variety of market structures that modern sellers are facing, we model the monopolistic firm as an upstream supplier who sells its product via a downstream market. The downstream market comprises an arbitrary number of retailers and various forms of market competition between the retailers, such as differentiated Cournot and Bertrand competition, no or full returns and collusions (cf. \Cref{tab:competition}).\footnote{The two-tier market model is an abstraction to capture the complexity of current markets. If we eliminate the second stage and assume that the supplier sells directly to the consumers, then our results still apply.} In all cases, retail demand is linear. In the first stage, the monopolistic firm sets a uniform price and in the second stage, the retailers observe the price and place their orders after the market demand has been realized. We assume that the supplier's capacity exceeds potential demand and that downstream retailers are symmetric. These assumptions serve the purpose to isolate the study of pricing decisions under uncertainty from various other strategic considerations such as stocking decisions, negotiation power, marketing and production \citep{Xue17,Don19}.\footnote{Under these assumptions, i.e., if the supplier's capacity exceeds potential demand and if the retailers are symmetric, then the single price scheme is optimal among a wide range of possible pricing mechanisms \citep{Har81,Ril83}. In addition, the symmetry of the retailers allows the study of purely competitive aspects which is not possible if retailers are heterogeneous \cite{Ty99}.}
%Under the assumption that retailers follow their equilibrium strategies (order quantities) in the downstream market, there exists a unique equilibrium quantity which is the same -- up to a constant -- in all considered market structures. This reduces all variations to essentially the same maximization problem for the supplier which corresponds to the setting of a monopolistic firm selling a product to a linear market with stochastic demand.

\paragraph{Results}
Our main theoretical contribution is that we characterize the seller's optimal prices as fixed points of the \emph{mean residual demand} (MRD) function of the stochastic demand level. Informally, if $\alpha, r$ denote the random demand level and the supplier's price respectively, then any optimal price, $r^*$, satisfies the equation $r^*=\m\(r^*\)$, where $\m\(r\):=\ex\(\alpha-br\mid \alpha>br\)$ and $b$ is a proper scaling constant (that can be normalized to $b=1$). The MRD function measures the expected additional demand given that demand has reached or exceeded a threshold $br$.\footnote{In reliability applications, this function is known as the mean residual life (MRL), see \cite{Sh07,Lax06,Be13,Be16}.} This characterization stems from the observation that the \emph{price elasticity of expected demand} can be expressed in terms of the MRD function as $r/\m\(r\)$, cf. equation \eqref{eq:elasticity}. Thus, optimal prices that correspond to points of unitary elasticity are fixed points of the MRD function. If $\m\(r\)/r$ is decreasing (or equivalently, if the price elasticity is increasing), then there exists a unique such optimal price. Both statements are presented in \Cref{thm:main}.\par
The above unimodality condition for the otherwise not necessarily concave nor quasi-concave seller's revenue function, strictly generalizes the well-known \emph{increasing generalized failure rate} (IGFR) condition \citep{La01,Be07}. Given the inclusiveness of the IGFR conditions, this suggests that \Cref{thm:main} applies to essentially most distributions that are commonly used in economic modeling \citep{Pa05,La06,Ba13}. The expressions of the price elasticity of expected demand and the seller's optimal price in terms of well-understood characteristics of the demand distribution (MRD functions) offer a novel perspective to the otherwise standard linear stochastic model.\footnote{While our results directly apply to the more general demand function of \cite{Lop19}, we stick to the linear model for expositional purposes. Except from the fact that linear markets have been consistently in the spotlight of economic research both due to their tractability and their accurate modeling of real situations, the study of the linear model is also technically motivated by \cite{Co16} who demonstrate that when information about demand is limited, firms may act efficiently \emph{as if} demand is linear. For practical purposes, this yields a simple and low regret pricing rule and provides a motivation to study linear markets in a systematic way.} They provide conditions that are easy to assess in practice and which can be useful to gain intuitive and economic interpretable results \citep{Be07}. In particular, these expressions provide a novel way to derive comparative statics on the response of the optimal price to various market characteristics (expressed as properties of the demand distribution) and to measure market performance. \par
The key intuition along this line, which is formally established in \Cref{lem:technical}, is that the seller's optimal price is higher in \emph{less elastic} markets which are precisely markets that can be ordered in the MRD \emph{stochastic order}. In other words, if two demand distributions can be ordered in terms of their MRD functions, then the supplier's optimal price is higher in the market with the dominating MRD function. Thus, the fact that elasticity is a critical factor in setting profitable prices and an important determinant of price changes in response to demand changes is formalized in terms of well-known demand characteristics. As a result, \Cref{lem:technical} is the key to leverage the theory of stochastic orders as a tool to compare prices in markets with different characteristics, such as market size and demand variability. Stochastic orders take into account various characteristics of the underlying demand distribution going thus, beyond summary statistics such as expectation or standard deviation, which provide a limited amount of information \citep{Sh07}.\par
In the comparative statics analysis, we start with the effect of market size in optimal prices and ask whether \emph{larger markets give rise to higher prices}. Our first finding is that stochastically larger markets do not necessarily lead to higher wholesale prices (\Cref{sub:larger}). Technically, this follows from the fact that the usual stochastic order does not imply (nor is implied by) the $\mrl$-order \citep{Sh07}. Hence, the intuition of \cite{La01} that \qt{size is not everything} and that prices are driven by different forces is justified by an appropriate theoretical framework. As \Cref{lem:technical} demonstrates, the correct criterion to order prices is the elasticity of different market instances rather than their size. In Theorem 4.3, we use this to derive a collection of demand transformations that preserve the MRD order (i.e., the elasticity) and hence the order of optimal prices. Returning to the effects of market size, one may still ask whether there exist conditions under which a probabilistic increase in market demand will lead to an increase in optimal prices. This question is addressed in \Cref{thm:size}, where we show that this is indeed the case whenever the initial demand is increased by a scale factor greater than one or (under some mild additional conditions) whenever an additional demand source is aggregated.\par
Next, we turn to the effect of demand variability. Does the seller charge higher/lower prices in more variable markets? The answer to this question again depends on the exact notion of variability that will be used. Under mild additional assumptions, \Cref{thm:var} gives two variability orders that preserve monotonicity: in both cases the seller charges a lower price in the less variable market. This conclusion remains true under the mean preserving transformation that is used by \cite{Li05,Li17}. However, as was the case with market size, the general statement that \emph{more variable markets give rise to higher prices} does not hold. In \Cref{sub:parametric}, we provide an example to show that this generalization fails in the standard case of parametric families of distributions that are compared in terms of their coefficient of variation, cf. \cite{La01}. All results from the comparative statics analysis are summarized in \Cref{tab:statics}.\par
We then turn to measure market performance and efficiency. Our main result in this direction, is a distribution-free (over the class of distributions with decreasing MRD function) upper bound  on the probability of a stockout, i.e., of no trade between the supplier and the retailers (\Cref{thm:notrade}). As shown in Examples~\ref{ex:exponential}, \ref{ex:beta} and \ref{ex:pareto}, this bound is tight and cannot be further generalized to distributions with increasing price elasticity of expected demand. In case that trade takes place, we measure market efficiency in terms of the realized profits and their distribution between the supplier and the retailers. Our results are summarized in \Cref{thm:share}. As intuitively expected (and in line with earlier results, cf. \cite{Ai12}), the supplier's profits are always higher if he is informed about the exact demand level and when retail competition is higher. However, there exists a range of intermediate demand realizations for which the supplier captures a larger share of the aggregate profits in the stochastic market. \par
Finally, we compare the aggregate realized profits between the deterministic and stochastic markets. The outcomes depend on the interplay between demand uncertainty and the level of retail competition. More specifically, there exists an interval of demand realizations for which the aggregate profits of the stochastic market are higher than the profits of the deterministic market. The interval reduces to a single point as the number of downstream retailers increases, but is unbounded in the case of $2$ retailers. In particular, for $n=1,2$, the aggregate profits of the stochastic market remain strictly higher than the profits of the deterministic market for all large enough realized demand levels. However, the performance of the stochastic market in comparison to the deterministic market degrades linearly in the number of competing retailers for demand realizations beyond this interval. This shows that uncertainty on the side of the supplier is more detrimental for the aggregate market profits when the level of retail competition is high, cf. \Cref{thm:agg}.

\subsection{Related Literature}\label{sub:literature}
The study of price-only contracts under demand uncertainty has been long motivated by \cite{Har81,Ril83} who show that committing to a single price is the optimal pricing strategy if the supplier's capacity exceeds potential demand and retailers are symmetric or if the monopolist is facing a known demand distribution. If the monopolist does not know the distribution of demand a-priori, as we assume in the present paper, then dispersed pricing improves upon the performance of a uniform price \cite{Da01}. By contrast, \cite{Ai12} show that wholesale price contracts are optimal even in the case of two competing chains. Compelling arguments for the linear pricing scheme are also provided in \cite{Ty99,Hw18} and references within. \cite{Pe07} and \cite{Be05} argue that apart from their practical prevalence, price-only contracts are relevant in modeling worst-case scenarios or interaction between sellers and buyers in cases with remaining uncertainty, i.e., after any efforts have been made to reduce the initial uncertainty through more elaborate schemes. This is further supported by \cite{Li17} who find that in a vertical market with a single manufacturer and a single retailer that is governed by a wholesale contract between them, less uncertainty may harm either or both members of the supply chain. This is partially the case also in the current model as we show in \Cref{sub:aggregate}. In particular, we find demand realizations for which the stochastic market outperforms the deterministic market in terms of aggregate profits. However, in our model, the supplier is always better off with reduced uncertainty whereas the retailers may be not.\par
The vertical market of a single supplier and multiple downstream competing retailers has been widely studied from different perspectives \citep{Pa97,Ty99,Ya06,Wu12}. Our results in \Cref{thm:share,thm:agg} are comparable with or mirror earlier findings in this line of literature. However, in the current setting, these findings complement our main results on the characterization of the demand elasticity in terms of the MRD function and the resulting comparative statics rather than being the main focus of our study. Concerning its other assumptions (market structure and timing of demand realization), our model enjoys similarities with the model of \cite{Wo97}. For extensive surveys of similar demand models, we refer to \cite{Hu13} and for earlier studies to \cite{Yao06}. Indicatively, the additive demand model with linear deterministic component is used by \cite{Pe99} and \cite{Ye16}. \par
Regarding the derived unimodality conditions, our findings are most closely related to \cite{La06,Be07} who derive the IGFR unimodality condition in the setting of one seller and one buyer. These are special cases of the present setting and, accordingly, the IGFR condition is a restriction of the unimodality condition in terms of the MRD function that we formulate in the current setting. IGFR distributions were first used in economic applications by \cite{Si76} and were popularized in the context of revenue management by \cite{La01}. Technical aspects of IGFR distributions are studied in \cite{Pa05,Ba13}. These results are more closely related to a companion paper \citep{Le18}, in which the authors focus on the technical properties of distributions that satisfy the current unimodality condition in terms of the MRD function.\footnote{Other closely related papers in the same direction include \cite{Leo20,Leon20,Leo21}, preliminary versions of which appear in \cite{Bel18,Kok18}.} The MRD function also arises naturally in many revenue management problems with demand uncertainty, see e.g., \cite{Ma18,Lu16,Co12,So09,So08} and \cite{Pe99} for a non-exhaustive list of related papers. However, to the best of our knowledge, there is no formal link between the elasticity of uncertain demand and the MRD function or the theory of stochastic orders in these papers. \par 
Similarities regarding the technical analysis can also be found between the current model and the literature on the price-setting newsvendor under stochastic demand, see e.g., \cite{Che04,Ky18,He18,Ko11,Ko18}. However, as the rest of the literature about the newsvendor, these results involve inventory considerations and hence are quite distinct from ours.\par
More closely related to the current methodology is the study of \cite{Xu10}. This paper is quite distinct from ours since it focus on a restricted set of stochastic orders and in a different newsvendor model that involves both pricing and stocking decisions. Concerning the results, \cite{Xu10} show that a stochastically larger demand leads to higher selling prices for the additive demand case. We extend these findings by showing that different notions of market size still lead to the same conclusion under certain conditions (\Cref{thm:transform,thm:size}) and by providing a case under which prices may be actually lower in a stochastically larger market, cf. \Cref{sub:larger}. Within similar contexts, \cite{La01,Li05,Kr10,Ch17} show that in general, optimal prices decrease as variability increases. Our current set of results refines these findings by using a wide range of stochastic orders that capture different notions of demand variability. Our findings suggest that increased demand variability may lead to both increased or decreased prices depending on the notion of variability that will be employed (cf. \Cref{sec:variability}). This demonstrates how various forms of knowledge about the demand distribution can be useful in the study of price movements and provides a theoretical explanation for empirically observed price changes in periods of turbulent market conditions.

\subsection{Outline}
The rest of the paper is structured as follows. In \Cref{sec:model,sec:main}, we define and analyze our model. \Cref{sec:statics} contains the comparative statics and \Cref{sec:performance} the study of market performance. \Cref{sec:conclusions} concludes the paper.

\section{The Model}\label{sec:model}
We consider a vertical market with a monopolistic upstream \emph{supplier} or seller, selling a homogeneous product (or resource) to $n=2$ downstream symmetric retailers who compete in a market with retail demand level $\alpha$\footnote{To ease the exposition, we restrict to $n=2$ retailers. As we show in \Cref{subn}, our results admit a straightforward generalization to arbitrary number $n$ of symmetric retailers.}. The supplier produces at a constant marginal cost which we normalize to zero. This corresponds to the situation in which the supplier's capacity exceeds potential demand by the retailers and the supplier's lone decision variable is his wholesale price, or equivalently his profit margin, $r\ge0$. \par
The supplier acts first (Stackelberg leader) and applies a linear pricing scheme without price differentiation, i.e., he chooses a unique wholesale price, $r\ge0$, for all retailers. We consider a market setting in which the supplier is less informed than the retailers about the \emph{retail demand level} $\alpha$. \footnote{We will refer to $\alpha$ throughout as the demand level. However, based on equation \eqref{demand}, $\alpha$ is also known as the choke or reservation price. Since, these constants are equivalent up to some transformation in our model, this should cause no confusion.} To model this, we assume that after the supplier's pricing decision but prior to the retailers' order decisions, a value for $\alpha$ is realized from a continuous (not-atomic) cumulative distribution function (cdf) $F$, with finite mean $\ex\alpha <\infty$ and nonnegative values, i.e., $F\(0\)=0$. Equivalently, $F$ can be thought of as the supplier's belief about the demand level and, hence, about the retailers' willingness-to-pay his price. We write $\F:=1-F$ for the tail distribution of $F$ and $f$ for its probability density function (pdf) whenever it exists. The support of $F$ is denoted by $S$, with lower bound $L=\sup{\{r\ge0: F\(r\)=0\}}$ and upper bound $H=\inf{\{r\ge0: F\(r\)=1\}}$ such that $0\le L\le H\le\infty$. We don't make any additional assumption about $S$: in particular, it may or may not be an interval. The case $L=H$ is not excluded\footnote{Formally, this case contradicts the assumption that $F$ is continuous or non-atomic. It is only allowed to avoid unnecessary notation and should cause no confusion.} and corresponds to the situation in which the supplier is completely informed about the retail demand level. \par
Given the demand realization $\alpha$, the aggregate quantity $q\(r\mid \alpha\):=\sum_{i=1}^2q_i\(r\mid \alpha\)$ that the retailers will order from the supplier is a function of the posted wholesale price $r$. Assuming risk neutrality, the supplier aims to maximize his expected profit function $\Pi_s$, which is equal to
\begin{equation}\label{suppliera}\Pi_s\(r\)=r\cdot\ex_{\alpha}q\(r\mid \alpha\).\end{equation}
The quantity $q\(r\mid\alpha\)$ depends on the form of second stage competition between the retailers. In this paper, we focus on markets with linear demand as in \cite{Mi59,Pe99,Hu13} and \cite{Co16} among others, and allow for a wide range of competition structures between the retailers (\Cref{tab:competition}). All these structures give rise -- in equilibrium -- to the same (up to a scaling constant) functional form for $q\(r\mid\alpha\)$ and hence to the same mathematical expression for the supplier's objective function. More importantly, in all these structures, the second-stage equilibrium between the retailers is unique and hence, $q\(r\mid\alpha\)$ is uniquely determined under the assumption that the retailers follow their equilibrium strategies in the game induced by each wholesale price $r\ge0$ (subgame perfect equilibrium). Specifically, we assume that each retailer $i$ faces the inverse demand function
\begin{equation}\label{demand}p_i=\alpha-\beta q_i-\gamma q_j,\end{equation}
for $j=3-i$ and $i=1,2$. Here, $\alpha/\(\beta+\gamma\)$ denotes the potential market size (primary demand), $\beta/\(\beta^2-\gamma^2\)>0$ the store-level factor and $\gamma/\beta$ the degree of product differentiation or substitutability between the retailers \citep{Si84,Wu12}. As usual, we assume that $\beta\le\gamma$. Each retailer's only cost is the wholesale price $r\ge0$ that she pays to the supplier. Hence, each retailer aims to maximize her profit function $\Pi_i$, which is equal to
\begin{equation}\label{retailer}\Pi_i\(q_i,q_j\)=q_i\(p_i-r\).\end{equation}
Given the demand realization $\alpha$, the equilibrium quantities $q_i^*:=q^*_i\(r\mid \alpha\)$ that maximize $\Pi_i$ for $i=1,2$ are given for various retail market structures in \Cref{tab:competition} as functions of the wholesale price $r$. Here, $\(\alpha-r\)_+$ denotes the positive part, i.e., $\(\alpha-r\)_+:=\max{\{0,\alpha-r\}}$. The assumption of no uncertainty on the side of retailers about the demand level $\alpha$ implies that $q_i^*$ corresponds both to the quantity that each retailer orders from the supplier and to the quantity that she sells to the market.
  
\begin{table*}[!htb]
\centering
\setlength{\tabcolsep}{3pt}
\renewcommand{\arraystretch}{1.4}
\arrayrulecolor{pcolor}
\begin{tabular}{ll}
\clrtwo
\textbf{Retail market structure} & \textbf{Retailer $i$'s equilibrium order}\\
\midrule[0.3mm]
\clrone
\cite{Si84}\hspace*{165pt}&\\\\[-0.5cm]
Cournot competition -- product differentiation & $q_i^*=\frac1{2\beta+\gamma}\(\alpha-r\)_+$ \\
Bertrand competition -- product differentiation & $q_i^*=\frac{\beta}{\(2\beta-\gamma\)\(\beta+\gamma\)}\(\alpha-r\)_+$ \\[0.15cm]
\midrule[0.3mm]
\clrone
\cite{Pa97}&\\\\[-0.5cm]
Single retailer no/full returns & $q^*=\frac1{2\beta}\(\alpha-r\)_+$ \\
Competing retailers (orders/price) -- no returns & $q_i^*=\frac1{2\beta+\gamma}\(\alpha-r\)_+$ \\
Competing retailers (orders/price) -- full returns & $q_i^*=\frac{\beta}{\(2\beta-\gamma\)\(\beta+\gamma\)}\(\alpha-r\)_+$\\[0.15cm]
\clrone
\midrule[0.3mm]
\cite{Ya06}&\\\\[-0.5cm]
Collusion between retailers -- product differentiation & $q_i^*=\frac1{\beta+\gamma}\(\alpha-r\)_+$\\[0.1cm]
\bottomline
\end{tabular}
\caption{Second-stage market structures (first column) and quantities, $q_i^*$, of retailer $i$ ordered from the supplier in equilibrium (second column). The equilibrium quantities are equal -- up to a multiplicative constant -- and unique which implies that the supplier's maximization problem is essentially the same for any of these forms of downstream competition.}
\label{tab:competition}
\end{table*}
The standard Cournot and Betrand outcomes arise as special cases of the above. In particular, for $\gamma=0$, the goods are independent and the monopoly solution $q_i^*=\frac1{2\beta}\(\alpha-r\)_+$ for $i=1,2$ prevails. For $\gamma=\beta>0$, the goods are perfect substitutes with $q^*_i=\frac1{2\beta}\(\alpha-r\)_+$ in Bertrand competition (at zero price) and $q^*_i=\frac{1}{3\beta}\(\alpha-r\)_+$ in Cournot competition for $i=1,2$. All of the above are assumed to be common knowledge among the participants in the market (the supplier and the retailers).

\section{Equilibrium analysis: supplier's optimal wholesale price}\label{sec:main}
We restrict attention to subgame perfect equlibria of the extensive form, two-stage game.\footnote{Technically, these are perfect Bayes-Nash equilibria, since the supplier has a belief about the retailers' types, i.e., their willingness-to-pay his price, that depends on the value of the stochastic demand parameter $\alpha$.} Assuming that at the second stage, the retailers play their unique equilibrium strategies $\(q_1^*,q_2^*\)$, then, according to \eqref{suppliera}, the supplier will maximize $\Pi^*_s\(r\)=r\cdot \ex_{\alpha}q^*\(r\mid \alpha\)$. For the competition structures of \Cref{tab:competition}, $q^*\(r\mid\alpha\)$ has the general form $q^*\(r\mid\alpha\)=\lambda_M \(\alpha-r\)_+$, where $\lambda_M>0$ is a suitable model-specific constant. Thus, at equilibrium, the supplier's expected profit maximization problem becomes
\begin{equation}\label{eq:supplier}\max_{r\ge0}{\Pi^*_s\(r\)}=\lambda_M\cdot \max_{r\ge0}{r\ex\(\alpha-r\)_+}.\end{equation}
From the supplier's perspective, we are interested in finding conditions such that the maximization problem in \eqref{eq:supplier} admits a unique and finite optimal wholesale price, $r^*\ge0$.

\begin{remark}
The vertical market structure is not necessary for our analysis to hold. In fact, if we eliminate the downstream market and instead assume that the firm sells directly to a market with inverse linear demand function $q\(r\)=\alpha-\beta r$, then our analysis still applies. This follows from the observation that in this case, the seller's expected profit maximization problem is 
\[\max_{r\ge0}\Pi_s\(r\)=\max_{r\ge0}r\mathbb E\(\alpha-\beta r\)_+\] 
which is the same as the maximization problem in equation \eqref{eq:supplier} after normalizing $\beta$ to $1$.
\end{remark}

\subsection{Deterministic Market}
First, we treat the case in which the supplier knows the primary demand $\alpha$ (deterministic market). According to the notation introduced in \Cref{sec:model}, this corresponds to the case $\alpha=L=H$. In this case $\Pi^*_s\(r\)=\lambda_M r\(\alpha-r\)_+$ and it is straightforward that $r^*\(\alpha\)=\alpha/2$. Hence, the complete information two-stage game has a unique subgame perfect Nash equilibrium, under which the supplier sells with optimal price $r^*\(\alpha\)=\alpha/2$ and each retailer orders quantity $q_i^*$ as determined by \Cref{tab:competition}.

\subsection{Stochastic Market}
The equilibrium behavior of the market in which the supplier does not know the demand level (stochastic market) is less straightforward. Now, $L<H$ and the supplier is interested in finding an $r^*$ that maximizes his expected profit in \eqref{eq:supplier}. For an arbitrary demand distribution $F$, $\Pi^*_s\(r\)$ may not be concave (nor quasi-concave) and, hence, not unimodal, in which case the solution to the supplier's optimization problem is not immediate. To obtain a general unimodality condition, we proceed by differentiating the supplier's revenue function $\Pi_s^*\(r\)$. First, since $\(\alpha-r\)_+$ is nonnegative, we write $\ex\(\alpha-r\)_+ = \int_{0}^{\infty}P\(\(\alpha-r\)_+ >u\) \du =\int_{r}^{\infty}\F\(u\)\du$, for $0 \le r<H$. Since $\ex\alpha<\infty$ and $F$ is non-atomic by assumption, we have that
\begin{align*}\frac{\d}{\d r}\ex\(\alpha-r\)_+=\frac{\d}{\d r}\(\ex\alpha-\int_{0}^{r}\F\(u\)\du\)=-\F\(r\)\end{align*}
for any $0<r<H$. With this formulation, both the supplier's revenue function and its first derivative can be expressed in terms of the \emph{mean residual demand} (MRD) function of $\alpha$. In general, the MRD function, $\m\(\cdot\)$, of a nonnegative random variable $\alpha$ with cumulative distribution function (cdf) $F$ and finite expectation, $\ex \alpha <\infty$, is defined as
\begin{equation}\label{eq:mrl}\m\(r\):=\ex\(\alpha-r \mid \alpha >r\)=\,\dfrac{1}{\vphantom{\tilde\F}\F\(r\)}\int_{r}^{\infty}\F\(u\)\du, \quad\mbox{for } r< H\end{equation}
 and $\m\(r\):=0$, otherwise, see, e.g., \cite{Sh07,Lax06} or \cite{Be16}\footnote{In this literature, the MRD function is known as the \emph{mean residual life} function due to its origins in reliability applications.}. Using this notation, we obtain that $\us=\lambda_Mr\m\(r\)\F\(r\)$ and 
\begin{equation}\label{derivative}\frac{\d \Pi^*_s}{\d r}\(r\)=\lambda_M\(\m\(r\)-r\)\F\(r\)=\lambda_Mr\(\frac{\m\(r\)}{r}-1\)\F\(r\)\end{equation}
for $0<r<H$. Based on \eqref{derivative}, the first order condition (FOC) for the supplier's revenue function is that $\m\(r\)=r$ or equivalently that $\m\(r\)/r=1$. We call the expression 
\begin{equation}\label{gmrl}\e\(r\):=\frac{\m\(r\)}{r}, \quad 0<r<H\end{equation} 
the \emph{generalized mean residual demand} (GMRD) function, see \cite{Le18}, due to its connection to the \emph{generalized failure rate (GFR)} function $\g\(r\):=rf\(r\)/\F\(r\)$, defined and studied by \cite{La99} and \cite{La01}. Its meaning is straightforward: while the MRD function $\m\(r\)$ at point $r>0$ measures the expected additional demand, given the current demand $r$, the GMRD function measures the expected additional demand as a percentage of the given current demand. Similarly to the GFR function, the GMRD function has an appealing interpretation from an economic perspective, since it is related to the \emph{price elasticity of expected or mean demand} (PEED), $\varepsilon\(r\)=-r\cdot\frac{\d}{\d r}\ex q\(r\mid\alpha\)/\ex q\(r\mid\alpha\)$ \citep{Xu10}. Specifically,
\begin{equation}\label{eq:elasticity}\e\(r\)=\frac{\m\(r\)}{r}=\(-\frac{-\F\(r\)}{\m\(r\)\vphantom{\tilde\F}\F\(r\)}\cdot r\)^{-1} = \(-r\cdot \frac{\frac{\d}{\d r}\ex \(\alpha-r\)_+}{\vphantom{\tilde\F}\ex\(\alpha-r\)_+}\)^{-1}=\varepsilon^{-1}\(r\)
\end{equation}
which implies that $\e\(r\)$ corresponds to the inverse of the \emph{price elasticity of expected demand} (PEED). Hence, in the current setting, demand distributions with decreasing GMRD, (DGMRD property), are precisely distributions that describe markets with increasing PEED, (IPEED property). This observation ties the economic property of IPEED to the distributional property of DGMRD. Accordingly, we will use the terms DGMRD and IPEED interchangeably.\par
Using \eqref{eq:elasticity}, the FOC in \eqref{derivative} asserts that the supplier's payoff is maximized at the point(s) of unitary elasticity. For an economically meaningful analysis, since realistic problems must have a PEED that eventually becomes greater than $1$ \citep{La06}, we give particular attention to distributions for which $\e\(r\)$ eventually becomes less than $1$, i.e., distributions for which $\bar{r}:=\sup{\{r\ge0:\m\(r\)\ge r\}}$ is finite. Observe that for a nonnegative random demand $\alpha$ with continuous distribution $F$ and finite expectation $\ex \alpha$, $\m\(0\)=\ex \alpha>0$ and hence $\bar{r}>0$. \par
Based on these considerations, it remains to derive conditions that guarantee the existence and uniqueness of an $r^*$ that satisfies the FOC and to show that this $r^*$ indeed corresponds to a maximum of the supplier's revenue function as given in \eqref{eq:supplier}. This is established in \Cref{thm:main} which is the main result of the present Section. 

\begin{theorem}[Equilibrium wholesale prices in the stochastic market]\label{thm:main}
Consider the supplier's maximization problem $\max_{r\ge0}{\Pi^*_s\(r\)}=\lambda_M\cdot \max_{r\ge0}{r\ex\(\alpha-r\)_+}$ and assume that the nonnegative demand parameter, $\alpha$, follows a continuous (non-atomic) distribution $F$ with support $S$ within $L$ and $H$. Then
\begin{itemize}[leftmargin=0cm,itemindent=.5cm,labelwidth=\itemindent,labelsep=0cm, align=left, itemsep=0cm, topsep=0pt]
\item[(a) Necessary condition: ] If an optimal price $r^*$ for the supplier exists, then $r^*$ satisfies the fixed point equation  
\begin{equation}\label{eq:fixed}r^*=\m\(r^*\).\end{equation} 
\item[(b) Sufficient conditions: ] If the generalized mean residual demand (GMRD) function, $\e\(r\):=\m\(r\)/r$, of $F$ is strictly decreasing and $\ex \alpha^2$ is finite, then at equilibrium, the supplier's optimal price $r^*$ exists and is the unique solution of \eqref{eq:fixed}. In this case, $r^*=\ex\alpha/2$, if $\ex\alpha/2 < L$, and $r^* \in \left[L, H\)$, otherwise.
\end{itemize}
\end{theorem}

\begin{proof}
(a) Since $\F\(r\)>0$ for $0<\alpha<H$, the sign of the derivative $\frac{\d\Pi^*_s}{\d r}\(r\)$ is determined by the term $\m\(r\)-r$ and any critical point $r^*$ satisfies $\m\(r^*\)=r^*$. Hence, the necessary part of the theorem is obvious from \eqref{derivative} and the continuity of $\frac{\d\Pi^*_s}{\d r}\(r\)$. (b) For the sufficiency part, it remains to check that such a critical point exists and corresponds to a maximum under the assumptions that $\e\(r\)$ is strictly decreasing and $\ex \alpha^2<\infty$. Clearly, $\m\(r\)-r$ is continuous and $\lim_{r\to 0_+}\m\(r\)-r=\ex\alpha>0$. Hence, $\us$ starts increasing on $\(0,H\)$. However, the limiting behavior of $\m\(r\)-r$ and hence of $\frac{\d\Pi^*_s}{\d r}\(r\)$ as $r$ approaches $H$ from the left, may vary depending on whether $H$ is finite or not. If $H$ is finite, i.e., if the support of $\alpha$ is bounded, then $\lim_{r\to H-}\(\m\(r\)-r\)=-H$. Hence, $\e\(r\)$ eventually becomes less than 1 and a critical point $r^*$ that corresponds to a maximum exists without any further assumptions. Strict monotonicity of $\e\(r\)$ implies that this $r^*$ is unique. If $H=\infty$, then an optimal solution $r^*$ may not exist because the limiting behavior of $\m\(r\)$ as $r\to\infty$ may vary, see \Cref{pareto} or \cite{Br03}. In this case, the condition of finite second moment ensures that $\bar{r}<\infty$. In particular, as shown in \cite{Le18}, if the GMRD function $\e\(r\)$ of a random variable $\alpha$ with unbounded support is decreasing, then $\lim_{r\to\infty}\e\(r\)<1$ if and only if $\ex \alpha^2$ is finite. This establishes existence. Uniqueness follows again from strict monotonicity of $\e\(r\)$ which precludes intervals of the form $\m\(r\)=r$ that give rise to multiple optimal solutions. \par
To prove the second claim of the sufficiency part, note that $\ex\alpha<2L$ is equivalent to $\m\(L\)< L$. Then, the DGMRD property implies that $\m\(r\)<r$ for all $r>L$, hence $r^*< L$. In this case, $\m\(r^*\)=\ex\alpha-r^*$ and hence $r^*$ is given explicitly by $r^*=\ex\alpha/2$, which may be compared with the optimal $r^*$ of the complete information case. On the other hand, if $\ex\alpha \ge 2L$, then for all $r<L$, $\m\(r\)=\ex\alpha-r\ge 2L-r>r$ which implies that $r^*$ must be in $\left[L,H\)$.
\end{proof}

The economic interpretation of the sufficiency conditions in part (b) of \Cref{thm:main} is immediate. By \eqref{eq:elasticity}, demand distributions with the DGMRD property are precisely distributions that exhibit \emph{increasing PEED} (IPEED property). In turn, finiteness of the second moment is required to ensure that the expected demand will eventually become elastic, even in the case of unbounded support, see \cite{Le18}, Theorem 3.2. Thus, part (b) characterizes in terms of their mathematical properties demand distributions that model linear markets with monotone and eventually elastic expected demand. These conditions apply to distributions that may neither be absolutely continuous (do not possess a density) nor have a connected support. 

\begin{remark}\label{strict} In the statement of \Cref{thm:main}, strict monotonocity can be relaxed to weak monotonicity without significant loss of generality. This relies on the explicit characterization of distributions with MRD functions that contain linear segments which is given in Proposition 10 of \cite{Ha81}. Namely, $\m\(r\)=r$ on some interval $J=[a,b]\subseteq S$ if and only if $\F\(r\)r^2=\F\(a\)a^2$ for all $r\in J$. If $J$ is unbounded, this implies that $\alpha$ has the Pareto distribution on $J$ with scale parameter $2$. In this case, $\ex \alpha^2=\infty$, see \Cref{pareto}, which is precluded by the requirement that $\ex\alpha^2<\infty$. Hence, to replace strict by weak monotonicity -- but still retain equilibrium uniqueness -- it suffices to exclude distributions that contain intervals $J=[a,b]\subseteq S$ with $b<\infty$ in their support, for which $\F\(r\)r^2=\F\(a\)a^2$ for all $r\in J$.
\end{remark}

\begin{example}[Pareto distribution]\label{pareto} The Pareto distribution is the unique distribution with constant GMRD and GFR functions over its support. Let $\alpha$ be Pareto distributed with pdf $f\(r\)=kL^kr^{-\(k+1\)}\mathbf{1}_{\{L\le r\}}$, and parameters $0<L$ and $k > 1$ (for $0<k\le 1$ we get $\ex\alpha=\infty$, which contradicts the basic assumptions of our model). To simplify, let $L=1$, so that $f\(r\)=k r^{-k-1}\mathbf{1}_{\{1 \leq r < \infty\}}$, $F\(r\) = \(1 - r^{-k}\)1_{\{1 \leq r < \infty\}}$, and $\ex\alpha=\frac{k}{k-1}$. The mean residual demand of $\alpha$ is given by $\m\(r\)=\frac{r}{k-1}+\frac{k}{k-1}\(1-r\)\mathbf 1_{\{0\le r<1\}}$ and, hence, is decreasing on $\left[0, 1\)$ and increasing on $\left[1, \infty\)$. However, the GMRD function $\e\(r\)=\m\(r\)/r$ is decreasing for $0<r<1$ and is constant thereafter, hence, $\alpha$ is DGMRD. Similarly, for $1\le r$ the failure (hazard) rate $\h\(r\)=kr^{-1}$ is decreasing, but the generalized failure rate $\g\(r\)=k$ is constant and, hence, $\alpha$ is IGFR. The payoff function of the supplier is 
\begin{equation*}\us=\lambda_Mr \, \m\(r\)\F\(r\)=\frac{\lambda_M}{\(k-1\)}\begin{cases}r \(r-rk+k\) & \mbox{if $0 \leq r < 1$}\\
r^{2-k} & \mbox{if $r \ge 1$}, \end{cases} \end{equation*}
which diverges as $r\to \infty$, for $k < 2$ and remains constant for $k=2$. In particular, for $k\le 2$, the second moment of $\alpha$ is infinite, i.e., $\ex \alpha^2=\infty$, which shows that for DGMRD distributions, the assumption that the second moment of $F$ is finite may not be dropped for part (b) of \Cref{thm:main} to hold. On the other hand, for $k > 2$, we get $r^* = \frac{k}{2\(k-1\)}$ as the unique optimal wholesale price, which is indeed the unique fixed point of $\m\(r\)$.
\end{example}

\subsection{General case with \texorpdfstring{$n$}{j} identical retailers}\label{subn}
To ease the exposition, we restricted our presentation to $n=2$ identical retailers. However, the present analysis applies to arbitrary number $n\ge 1$ of symmetric retailers for all competition-structures that give rise to a unique second-stage equilibrium in which the aggregate ordered quantity depend on $\alpha$ via the term $\(\alpha-r\)_+$ as in \Cref{tab:competition}. This relies on the fact, that in such markets, the total quantity that is ordered from the supplier depends on $n$ only up to a scaling constant. Thus, the approach to the supplier's expected profit maximization in the first-stage remains the same independently of the number of second-stage retailers. To avoid unnecessary notation, we present the general case for the classic Cournot competition. \par
Formally, let $N=\{1,2, \ldots, n\}$, with $n\ge 1$ denote the set of symmetric retailers. A strategy profile (retailers' orders from the supplier) is denoted by $\mathbf q=\(q_1, q_2, \ldots, q_n\)$ with $q=\sum_{j=1}^nq_j$ and $q_{-i}:=q-q_i$. Assuming linear inverse demand function $p=\alpha-\beta q$, the payoff function of retailer $i$, for $i\in N$, is given by $\Pi_i\(q, q_{-i}\)=q_i\(p-r\)$. Under these assumptions, the second stage corresponds to a linear Cournot oligopoly with constant marginal cost, $r$. Hence, each retailer's equilibrium strategy, $q^*_i\(r\)$, is given by $q^*_i\(r\)= \frac{1}{\beta\(n+1\)}\(\alpha-r\)_+$, for $r\ge 0$. Accordingly, in the first stage, the supplier's expected revenue function on the equilibrium path is given by $\Pi^*_s\(r\)=rq^*\(r\)=\frac{n}{\beta\(n+1\)}r\ex\(\alpha-r\)_+$. Hence, it is maximized again at $r^*\(\alpha\)=\alpha/2$ if the supplier knows $\alpha$ or at $r^*=\m\(r^*\)$ if the supplier only knows the distribution $F$ of $\alpha$. Based on the above, the number of second-stage retailers affects the supplier's revenue function only up to a scaling constant and \Cref{thm:main} is stated unaltered for any $n\ge 1$.

\section{Comparative Statics}\label{sec:statics}
The main implication of the closed form expression of the supplier's optimal price in terms of the MRD function (equation  \eqref{eq:fixed}) is that it facilitates a comparative statics analysis via the rich theory of \emph{stochastic orders} \citep{Sh07,Lax06,Be16}. Since the equilibrium quantity and price $q^*,p^*$ are both monotone in the wholesale price $r^*$, our focus will be on $r^*$ as the demand distribution characteristics vary. To obtain a meaningful comparison between different market instances (i.e., instances of the same market that correspond to different demand distributions), we assume throughout equilibrium uniqueness and hence, unless stated otherwise, we consider only distributions for which \Cref{thm:main} applies\footnote{Since the DGMRD property is satisfied by a very broad class of distributions, see \cite{Ba13}, \cite{Ko11} and \cite{Le18}, we do not consider this as a significant restriction. Still, since it is sufficient (together with finitenes of the second moment) but not necessary for the existence of a unique optimal price, the analysis naturally applies to any other distribution that guarantees equilibrium existence and uniqueness.}. First, we introduce some additional notation. \par
Let $X_1\sim F_1,X_2\sim F_2$ be two nonnegative random variables -- or equivalently demand distributions -- with supports between $L_1$ and $H_1$ and $L_2$ and $H_2$, respectively (cf. definition of $L$ and $H$ in \Cref{sec:model}) and MRD functions $\m_1\(r\)$ and $\m_2\(r\)$. We say that $X_1$ is less than $X_2$ in the \emph{mean residual demand} order, denoted by $X_1\mrl X_2$\footnote{In reliability applications, the MRD-order is commonly known as the mean residual life (MRL)-order, \cite{Sh07}.}, if $\m_1\(r\)\le \m_2\(r\)$ for all $r\ge0$. This order plays a key role in the present model. Specifically, by \eqref{eq:elasticity}, we have that $\m_1\(r\)\le \m_2\(r\)$ for any $r\ge0$ if and only if $\varepsilon_2\(r\)\le \varepsilon_1\(r\)$ for any $r\ge0$, i.e., if and only if the price elasticity of expected demand in market $X_2$ is less than the price elasticity of expected demand in market $X_1$ for any wholesale price $r\ge0$. This motivates the following definition.

\begin{definition}\label{def:elastic}
We will say that market $X_2$ is \emph{less elastic} than market $X_1$, denoted by $X_2\el X_1$, if $\varepsilon_2\(r\)\le \varepsilon_1\(r\)$ for every $r\ge0$. Based on the above, $X_2\el X_1$ if and only if $X_1\mrl X_2$. 
\end{definition}
Using this notation, the following Lemma captures the importance of the characterization of the optimal price via the fixed point equation \eqref{eq:fixed}. 

\begin{lemma}\label[lemma]{lem:technical} Let $X_1\sim F_1, X_2\sim F_2$ be two nonnegative, continuous and strictly DGMRD demand distributions with finite second moments. If $X_2$ is less elastic than $X_1$, then the supplier's optimal wholesale price is lower in market $X_1$ than in market $X_2$. In short, if $X_2\el X_1$, then $r^*_1\le r^*_2$. 
\end{lemma}
\begin{proof} By definition, $X_2\el X_1$ implies that $\varepsilon_2\(r\)\le \varepsilon_1\(r\)$ for every $r\ge0$ which by \eqref{eq:elasticity} is equivalent to $\e_1\(r\)\le \e_2\(r\)$ for all $r\ge0$. Hence, by \eqref{eq:fixed}, $1=\e_1\(r_1^*\)\le \e_2\(r_1^*\)<\e_2\(r\)$ for all $r<r^*_1$, where the second inequality follows from the assumption that $X_2$ is strictly DGMRD. Since $\e_2\(r^*_2\)=1$, it follows that $r^*_1\le r^*_2$.
\end{proof}
\Cref{lem:technical} states that the supplier charges a higher price in a \emph{less elastic} market. Although trivial to prove once \Cref{thm:main} has been established, it is the key to the comparative statics analysis in the present model. Indeed, combining the above, the task of comparing the optimal wholesale price $r^*$ for varying demand distribution parameters -- such as market size or demand variability -- essentially reduces to comparing demand distributions (market instances) in terms of their elasticities or equivalently in terms of their MRD functions. Such conditions can be found in \cite{Sh07} and \cite{Be13} and provide the framework for the subsequent analysis.

\subsection{Transformations that preserve the MRD-order}
\Cref{lem:technical} provides a natural starting point to study the response of the equilibrium wholesale price, $r^*$, to changes in the demand distribution. In particular, if a change in the demand distribution preserves the $\mrl$-order, then \Cref{lem:technical} readily implies that this change will also preserve the order of wholesale prices. Specifically, let $X_1\sim F_1,X_2\sim F_2$ denote two different demand distributions, such that $X_1\mrl X_2$. In this case, we know by \Cref{lem:technical} that $r^*_1\le r^*_2$. We are interested in determining transformations of $X_1, X_2$ that preserve the $\mrl$-order and hence the ordering $r^*_1\le r^*_2$. Unless otherwise stated, we assume that the random demand is such that it satisfies the sufficiency conditions of \Cref{thm:main} and hence that the supplier's optimal wholesale price exists and is unique.
\begin{theorem}\label{thm:transform} Let $X_1\sim F_1,X_2\sim F_2$ denote two nonnegative, continuous and strictly DGMRD demand distributions, with finite second moments, such that $X_1\mrl X_2$.
\begin{enumerate}[label=(\roman*), itemsep=0cm]
\item If $Z$ is a nonnegative, IFR distribution, independent of $X_1$ and $X_2$, then $\mathbf{r^*_{X_1+Z}}\le \mathbf{r^*_{X_2+Z}}$.
\item If $\phi$ is an increasing, convex function, then $r^*_{\phi\(X_1\)}\le r^*_{\phi\(X_2\)}$. 
\item If $X_p\sim p F_1+\(1-p\)F_2$ for some $p\in \(0,1\)$, then $r^*_{X_1}\le \mathbf{r^*_{X_p}}\le r^*_{X_2}$.
\end{enumerate}
\end{theorem}
\begin{proof} Part (i) follows from Lemma 2.A.8 of \cite{Sh07}. Since the resulting distributions $X_i+Z, i=1,2$ may not be DGMRD nor DMRD, the setwise notation is necessary. Part (ii) follows from Theorem 2.A.19 (ibid). Equilibrium uniqueness is retained in the transformed markets, $\phi\(X_i\), i=1,2$, since the DGMRD class of distributions is closed under increasing, convex transformations, see \cite{Le18}. Finally, part (iii) follows from Theorem 2.A.19. However, the DGMRD class is not closed under mixtures and hence, in this case, the $X_p$ market may have multiple equilibria, which necessitates, as in part (i), the setwise statement for the wholesale equilibrium prices of the $X_p$ market.
\end{proof}
\begin{remark}
\cite{Mu90} show that the strict $\mrl$-order -- i.e., if the inequality $\m_1\(r\)<\m_2\(r\)$ is strict for all $r$ -- is closed under monotonically non-decreasing transformations and closed in a reversed sense under monotonically non-increasing transformations. The $\mrl$-order is also closed under convolutions, provided that the convoluting distribution has log-concave density (as is the case with many commonly used distributions, \cite{Ba05}), \cite{Mu92}. Finally, if instead of $X_1\mrl X_2$, $X_1$ and $X_2$ are ordered in the stronger \emph{hazard rate} order, i.e., if $\h_1\(r\)\le \h_2\(r\)$ for all $r\ge0$, denoted by $X_1\hr X_2$, then part (i) of \Cref{thm:transform} remains true by Lemma 2.A.10 of \cite{Sh07}, even if $Z$ is merely DMRD (instead of IFR). 
\end{remark}

\subsection{Market size}\label{size}
Next, we turn to demand transformations that intuitively correspond to larger market instances. Again, to avoid unnecessary technical complications, we will restrict attention to demand distributions which satisfy the sufficiency conditions of \Cref{thm:main} (e.g., DMRD or DGMRD distributions).

\subsubsection{Stochastically larger markets}\label{sub:larger}
Our first finding in this direction is that stochastically larger markets do not necessarily lead to higher wholesale prices. Technically, this follows from the fact that the stochastic order does not imply (nor is implied by) the $\mrl$-order \citep{Sh07}. In particular, \Cref{lem:technical} demonstrates that the correct criterion to order prices is the elasticity rather than the size of different market instances. This provides a theoretical explanation for the intuition of \cite{La01} that \qt{size is not everything} and that prices are driven by different forces. \par
Formally, let $X\sim F, Y\sim G$ denote two market instances. If $\G\(r\)\le \F\(r\)$ for all $r\ge0$, then $Y$ is said to be less than $X$ in the usual \emph{stochastic order}, denoted by $Y\st X$. It is immediate that $Y\st X$ implies $\ex Y\le \ex X$. The following example, adapted from \cite{Sh91}, shows that wholesale prices can be lower in a stochastically larger market instance. Specifically, let $X\sim F$ be uniformly distributed on $[0,1]$ and let $Y\sim G$ have a piecewise linear distribution with $G\(0\)=0, G\(1/3\)=7/9, G\(2/3\)=7/9$ and $G\(1\)=1$. Then, as shown in \Cref{stochuni}, $Y\st X$ (right panel) but $r^*_X\le r^*_Y$ (left panel).
\begin{figure*}[htp!]
\centering
\includegraphics[width=\linewidth]{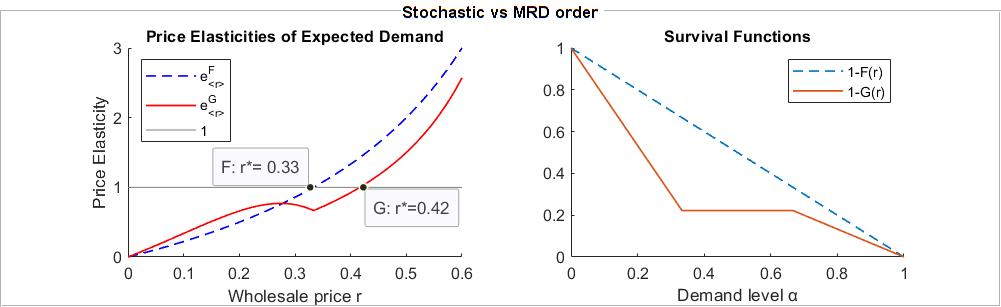}
\caption{$F$ stochastically dominates $G$ (right panel), yet $r^*_F<r^*_G$ (left panel). The optimal prices $r^*_F$ and $r^*_G$ are at the points of unitary elasticity, i.e., at the intersection of the price elasticities curves with the horizontal line at $1$.}
\label{stochuni}
\end{figure*}

\subsubsection{Reestimating demand} 
The above example suggests that the statement \emph{larger markets lead to higher prices} cannot be obtained in full generality. This brings us to the main part of this section which is to investigate conditions under which an increase in market demand leads to an increase in optimal prices. Formally, let $X$ denote the random demand in an instance of the market under consideration. Let $c\ge1$ denote a positive constant and $Z$ an additional random source of demand that is independent of $X$. Moreover, let $r^*_{X}$ denote the equilibrium wholesale price in the initial market and $r^*_{cX}, r^*_{X+Z}$ the equilibrium wholesale prices in the markets with random demand $cX$ and $X+Z$ respectively. How does $r^*_{X}$ compare to $r^*_{cX}$ and $r^*_{X+Z}$?\par
While the answer for $r^*_{cX}$ is rather straightforward, see \Cref{thm:size} below, the case of $X+Z$ is more complicated. Specifically, since DGMRD random variables are not closed under convolution, see \cite{Le18}, the random variable $X+Z$ may not be DGMRD. This may lead to multiple equilibrium wholesale prices in the $X+Z$ market, irrespectively of whether $Z$ is DGMRD or not. To deal with the possible multiplicity of equilibria, we will write $\mathbf {r^*_{W}}:=\left\{r: r=\m_{W}\(r\)\right\}$ to denote the set of all possible equilibrium wholesale prices. Here, $\m_{W}$ denotes the MRD function of a $W\sim F_W$ demand distribution, e.g., $W:=X+Z$. To ease the notation, we will also write $\mathbf{r^*_W}\le\mathbf{r^*_V}$, when all elements of the set $\mathbf{r^*_W}$ are less or equal than all elements of the set $\mathbf{r^*_V}$.\par
\Cref{thm:size} 
conforms wite prices are always higher in the larger $cX$ market and under some additional conditions also in the $X+Z$ market.
\begin{theorem}\label{thm:size} Let $X\sim F$ be a nonnegative and continuous demand distribution with finite second moment.
\begin{enumerate}[label=(\roman*), noitemsep]
\item If $X$ is DGMRD and $c\ge 1$ is a positive constant, then $r^*_{X}\le r^*_{cX}$.
\item If $X$ is DMRD and $Z$ is a nonnegative, continuous demand distribution with finite second moment and independent of $X$, then $r^*_{X}\le \mathbf{r^*_{X+Z}}$, i.e., $r^*_X\le r^*_{X+Z}$ for any equilibrium wholesale price $r^*_{X+Z}$ of the $X+Z$ market.
\end{enumerate}
\end{theorem}
\begin{proof} The proof of part (i) follows directly from the preservation property of the $\mrl$-order that is stated in Theorem 2.A.11 of \cite{Sh07}. Specifically, since $\m_{cX}\(r\)=c\m_X\(r/c\)$ is the MRD function of $cX$, we have that $\m_{cX}\(r\)=r\cdot\frac{\m_X\(r/c\)}{r/c}=r\cdot\e\(r/c\)\ge r\cdot\e\(r\)=\m_X\(r\)$, for all $r>0$, with the inequality following from the assumption that $X$ is DGMRD. Hence, $X \mrl cX$ or equivalently $cX\el X$, cf. \Cref{def:elastic}, which by \Cref{lem:technical} implies that $r^*_X\le r^*_{cX}$. \par
Part (ii) follows from Theorem 2.A.11 of \cite{Sh07}. The proof necessitates that $X$ is DMRD and hence requiring that $X$ is merely DGMRD is not enough. Since, $X$ is DMRD, we know that $r<\m_X\(r\)$ for all $r<r^*_{X}=\m_X\(r^*_{X}\)$. Together with $X\mrl X+Z$, this implies that $r<\m_X\(r\)\le \m_{X+Z}\(r\)$, for all $r<r^*_{X}$.  Hence, $\mathbf {r^*_{X+Z}}\subseteq[r^*_X, \infty)$, which implies that in this case, $r^*_{X}$ is a lower bound to the set of all possible wholesale equilibrium prices in the $X+Z$ market.
\end{proof}

\subsection{Market demand variability}\label{sec:variability}
The response of the equilibrium wholesale price to increasing (decreasing) demand variability is even less straightforward. There exist several stochastic orders that compare random variables in terms of their variability and the effects on prices largely depend on the exact order that will be employed. To proceed, we first introduce some additional notation.

\subsubsection{Variability or dispersive orders}\label[paragraph]{varordisp}
Let $X_1\sim F_1$ and $X_2\sim F_2$ be two nonnegative distributions with equal means, $\ex X_1=\ex X_2$, and finite second moments. If $\int_{r}^{\infty}\F_1\(u\)\du\le \int_{r}^{\infty}\F_2\(u\)\du$ for all $r\ge0$, then $X_1$ is said to be smaller than $X_2$ in the \emph{convex order}, denoted by $X_1\cx X_2$. If $F_1^{-1}$ and $F_2^{-1}$ denote the right continuous inverses of $F_1,F_2$ and $F_1^{-1}\(r\)-F_1^{-1}\(s\)\le F_2^{-1}\(r\)-F_2^{-1}\(s\)$ for all $0<r\le s<1$, then $X_1$ is said to be smaller than $X_2$ in the \emph{dispersive order}, denoted by $X_1\disp X_2$. Finally, if $\int_{F_1^{-1}\(p\)}^{\infty} \F_1\(u\)\du \le \int_{F_2^{-1}\(p\)}^{\infty} \F_2\(u\)\du$ for all $p\in \(0,1\)$, then $X_1$ is said to be smaller than 
$X_2$ in the \emph{excess wealth order}, denoted by $X_1\ew Y$. \cite{Sh07} show that $X\disp Y \implies X\ew Y \implies X\cx Y$ which in turn implies that $\var\(X\)\le \var\(Y\)$. 
\paragraph{Does less variability imply a lower (higher) wholesale price?}
The answer to this question largely depends on the notion of variability that we will employ. \cite{Xu10} use the more general $\cx$-order to conclude that under mild additional assumptions, less variability implies higher prices. Concerning the present setting, ordering two demand distributions $X\sim F$ and $Y\sim G$ in the $\cx$-order does not in general suffice to conclude that wholesale prices in the $X$ and $Y$ markets are ordered respectively. This is due to the fact that the $\cx$-order does not imply the $\mrl$-order. An illustration is provided in \Cref{convex1,convex2}. 

\begin{figure*}[!htb]
\includegraphics[width=\linewidth]{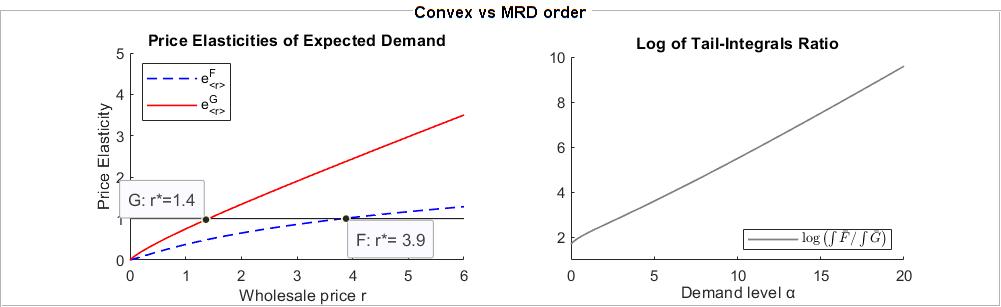}
\caption{Comparison of $X\sim F$, where $F$ denotes a $\text{Lognormal}\(\mu=0.5,\sigma=1\)$ and $ Y\sim G$, where $G$ denotes a $\text{Gamma}\(\alpha=2, \beta=0.25\)$. $Y\cx X$ (right panel) and $r^*_Y< r^*_X$ (left panel).}
\label{convex1}
\end{figure*}

In \Cref{convex1}, we consider two demand distributions, $X\sim F$, a Lognormal $\(\mu=0.5,\sigma=1\)$ and $Y\sim G$, a Gamma $\(\alpha=2, \beta=0.25\)$. For this choice of parameters, $\ex X=\ex Y=0.5$ and hence $X, Y$ are ordered in the $\cx$-order if and only if the tail-integrals of $F$ and $G$ are ordered, see \cite{Sh07} Theorem 3.A.1. The right panel depicts the log of the ratio of these integrals, i.e., $\log{\(\int_{r}^{\infty}\F \du/\int_{r}^{\infty}\G\du\)}$ which remains throughout positive (and increasing). Hence, $Y\cx X$. The left panel depicts the price elasticities of expected demand in the $X$ and $Y$ markets. As can be seen, the supplier charges a higher price in the $X\sim F$ market than in the less variable (according to the $\cx$-order) $Y\sim G$ market. \par
The above conclusion is reversed in the case of \Cref{convex2}. In this example, we consider two demand distributions, $X\sim F$ with $F$, as above, a $\text{Lognormal}\(\mu=0.5,\sigma=1\)$ and $Y\sim G$, a Gamma $\(\alpha=8, \beta=0.25/4\)$. This choice of parameters retains the equality $\ex X=\ex Y=0.5$ and hence, $X, Y$ can be ordered in the $\cx$-order if and only if the tail-integrals of $F$ and $G$ can be ordered. Again, the right panel depicts the log of the ratio of these integrals which remains throughout positive (and increasing). Hence, $Y\cx X$. However, the picture in the left panel is now reversed. As can be seen, the supplier now charges a lower price in the $X\sim F$ market than in the less variable (according to the $\cx$-order) $Y\sim G$ market.
\begin{figure*}[!htb]
\includegraphics[width=\linewidth]{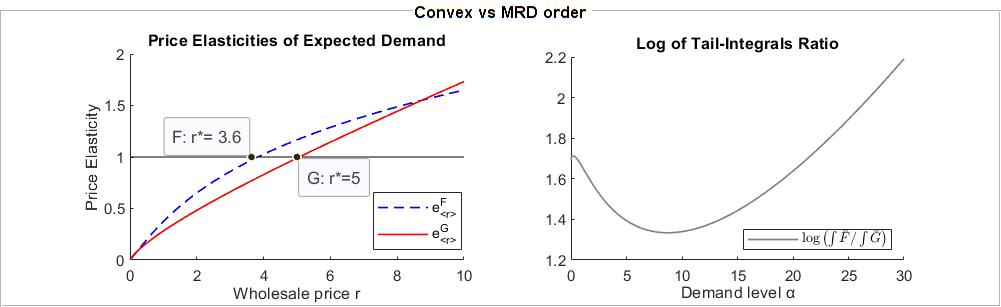}
\caption{Comparison of $X\sim F$, with $F$ $\text{Lognormal}\(\mu=0.5, \sigma=1\)$ and $Y \sim G$ with $G$ $\text{Gamma}\(\alpha=8, \beta=0.25/4\)$. $Y\cx X$ (right panel) and $r^*_X< r^*_Y$ (left panel).}
\label{convex2}
\end{figure*}

More can be said, if we restrict attention to the $\ew$- and $\disp$-orders. We will write $L_i$ to denote the lower end of the support of variable $X_i$ for $i=1,2$. 
\begin{theorem}\label{thm:var}
Let $X_1\sim F_1, X_2\sim F_2$ be two nonnegative, continuous, strictly DGMRD demand distributions with finite second moment. In addition,
\begin{enumerate}[label=(\roman*), noitemsep]
\item if either $X_1$ or $X_2$ are DMRD and $X_1\ew X_2$, and if $L_1\le L_2$, then $r^*_1\le r^*_2$.
\item if either $X_1$ or $X_2$ are IFR and $X_1\disp X_2$, then $r^*_1\le r^*_2$.
\end{enumerate}
\end{theorem}
\begin{proof}
The first part of \Cref{thm:var} follows directly from Theorem 3.C.5 of \cite{Sh07}. Based on its proof, the assumption that at least one of the two random variables is DMRD (and not merely DGMRD) cannot be relaxed. Part (ii) follows directly from Theorem 3.B.20 (b) of \cite{Sh07} and the fact that the $\hr$-order implies the $\mrl$-order. As in part (i), the condition that both $X_1$ and $X_2$ are DGMRD does not suffice and we need to assume that at least one is IFR. Recall, that $\text{IFR} \subset \text{DMRD} \subset \text{DGMRD}$ with all inclusions being strict, see e.g., \cite{Le18}. 
\end{proof}
The first implication of \Cref{thm:var} is that there exist classes of distributions for which less variability implies lower wholesale prices. This is in contrast with the results of \cite{La01} and \cite{Xu10} (for the additive demand case) and sheds light on the effects of upstream demand uncertainty. In these models, uncertainty falls to the retailer, and the supplier charges a higher price to capture an increasing share of all supply chain profits as variability reduces. Contrarily, if uncertainty falls to the supplier as in the present model, then the supplier may charge a lower price as variability increases. \par
The second implication is that these results, albeit general, do not apply to all distributions that are comparable according to \emph{some} variability order. As illustrated with the examples in \Cref{convex1,convex2} and the convex-order, less variability may lead to both higher or lower wholesale prices. From a managerial perspective, this implies that the effect of demand variability on prices crucially depends on the exact notion of variability that will be employed and may be ambiguous even under the standard setting of linear demand that is studied here.  

\subsubsection{Mean preserving transformation}
To further study the effects of demand variability, we use the mean preserving transformation $X_\kappa:=\kappa X+\(1-\kappa\)\mu$, where $\mu=\ex X$ and $\kappa\in[0,1]$, see \cite{Li05} and \cite{Li17}. Indeed, $\ex X_\kappa =\ex X$ and $\mathrm{Var}\(X_\kappa\)=\kappa^2\mathrm{Var}\(X\)\le \mathrm{Var}\(X\)$, i.e., $X_\kappa$ has the same mean and support as but is \qt{less variable} than $X$. \Cref{thm:kappa} shows that $X_\kappa \mrl X$ and hence, by \Cref{lem:technical} the supplier always sets a higher price in market $X$ than in the \enquote{less variable} market $X_\kappa$. This recovers in a straightforward way the finding of \cite{Li05}.

\begin{theorem}\label{thm:kappa}
Let $X\sim F$ be a nonnegative, continuous, DGMRD demand distribution with finite mean, $\mu$, and variance, $\sigma^2$, and let $X_\kappa:=\kappa X+\(1-\kappa\)\mu$, for $\kappa \in [0,1]$. Then, $X_\kappa \mrl X$ and $r^*_\kappa\le r^*$. 
\end{theorem}

\begin{proof}
It suffices to show that $Y\equiv \mu$ is smaller than $X$ in the mrd-order, i.e., that $Y\mrl X$. The conclusion then follows from Theorem 2.A.18 of \cite{Sh07} and \Cref{lem:technical}. In turn, to show that $Y\mrl X$, it suffices to show that $\int_{x}^{\infty}\F_X\(u\)\d u/\int_{x}^{\infty}\F_Y\(u\)\d u$ increases in $x$ over $\left\{x: \int_{x}^{\infty}\F_Y\(u\)\d u>0\right\}$, cf. \cite{Sh07} (2.A.3). Since, $\F_Y\(u\)=\mathbf{1}_{\{x<\mu\}}$, this is equivalent to showing that $\int_{x}^{\infty}\F_X\(u\)\d u/\(\mu-x\)$ increases in $x$ for $x<\mu$. Differentiating with respect to $x$ and reordering the terms, we obtain that the previous expression increases in $x$ for $x<\mu$ if and only if $\m_X\(x\)\ge \mu -x$ for $x\in [0,\mu)$. However, this is immediate, since $\m\(x\)\ge \int_{x}^{\infty}\F_X\(u\)\d u=\mu-\int_{0}^x\F_X\(u\)\d u\ge \mu- x$.  
\end{proof}

\subsubsection{Parametric families of distributions}\label{sub:parametric}
To elaborate on the fact that different variability notions may lead to different responses on wholesale prices, we consider the parametric approach of \cite{La01}. Given a random variable $X$ with distribution $F$, let $X_i:=\delta_i+\lambda_iX$ with $\delta_i\ge0$ and $\lambda_i>0$ for $i=1,2$. \cite{La01} show that in this case, the wholesale price is dictated by the coefficient of variation, $CV_i=\sqrt{\var\(X_i\)}/\ex X_i$. Specifically, if $CV_2<CV_1$, then $r^*_1<r^*_2$, i.e., in their model, a lower $CV$, or equivalently a lower relative variability, implies a higher price. This is not true for our model.\par
To see this, we consider two normal demand distributions $X_1\sim N\(\mu_1,\sigma_1^2\)$ and $X_2\sim N\(\mu_2,\sigma_2^2\)$. By Table 2.2 of \cite{Be16}, if $\sigma_1<\sigma_2$ and $\mu_1\le \mu_2$, then $X_1\mrl X_2$ and hence, by \Cref{lem:technical}, $r^*_1\le r^*_2$. However, by choosing $\sigma_i$ and $\mu_i$ appropriately, we can trivially achieve an arbitrary ordering of their relative variability in terms of their $CV$'s. The reason for this ambiguity is that changing $\mu_i$ for $i=1,2$, not only affects $CV_i$, i.e., the relative variability, but also the central location of the respective demand distribution. In contrast, under the assumption that $\ex X_1=\ex X_2$, the stochastic orders approach of the previous paragraph provides a more clear insight. The results of the comparative statics analysis are summarized in \Cref{tab:statics}.

\begin{table*}[!htb]
\centering
\setlength{\tabcolsep}{4pt}
\renewcommand{\arraystretch}{1.4}
\arrayrulecolor{pcolor}
\begin{tabular}{llr}
\clrtwo
\b Demand Transformations &&\\
\multicolumn{2}{l}{Assumptions on market demand distributions} & Optimal prices\\
\midrule[0.3mm]
%\clrone
$X_1\mrl X_2$ & \specialcell[t]{$Z\ge0$, IFR, independent of $X_1,X_2$} & $\mathbf{r^*_{X_1+Z}}\le \mathbf{r^*_{X_2+Z}}$\\
& $\phi\(x\)$ increasing and convex & $r^*_{\phi\(X_1\)}\le r^*_{\phi\(X_2\)}$\\
%\clrone
& $X_p\sim pF_1+\(1-p\)F_2$, $p\in\(0,1\)$\hspace{53pt} & $r^*_{X_1}\le r^*_{X_p}\le r^*_{X_2}$\\[0.1cm]\bottomrule\\[-0.6cm]
\clrtwo
\b Market Size &&\\
\multicolumn{2}{l}{Assumptions on market demand distributions} & Optimal prices\\
\midrule[0.3mm]
%\clrone
$X$ & $c\ge1$& $r^*_X\le r^*_{cX}$\\
$X$ DMRD & \specialcell[t]{$Z\ge0, \ex Z^2<\infty$, independent of $X$} & $r^*_X\le \mathbf{r^*_{X+Z}}$ \\
%\clrone
$X_1\st X_2$ && inconclusive\\[0.1cm]
\bottomrule\\[-0.6cm]
\clrtwo
\b Demand Variability&&\\
\multicolumn{2}{l}{Assumptions on market demand distributions} & Optimal prices\\
\midrule[0.3mm]
%\clrone
$X_1\cx X_2$ && inconclusive\\
$X_1\ew X_2$ & $L_1\le L_2$ and $X_1$ or $X_2$ DMRD & $r^*_1\le r^*_2$\\
%\clrone
$X_1\disp X_2$ & $X_1$ or $X_2$ IFR & $r^*_1\le r^*_2$\\
$X_i:=\delta_i+\lambda_i X$ & $CV_1\le CV_2, \delta_i\ge0, \lambda_i>0, i=1,2$ & inconclusive\\
%\clrone
$X_{\kappa}:=\kappa X+\(1-\kappa\)\ex X$ & $\kappa \in [0,1]$ & $r^*_{\kappa}\le r^*$\\
\bottomrule[0.3mm]
\end{tabular}
\caption{Summary of the main comparative statics results on the seller's optimal price. The first two columns, \qt{Assumptions on Market Demand} describe the market parameters, characteristics and their relationships (if any) and the third column, \qt{Optimal Prices}, compares the optimal prices in the markets in question. Bold values refer to transformations that yield non-unique optimal prices in which case the notationa \qt{$\le$} should be understood elementwise. If not stated otherwise, $X, X_1$ and $X_2$ satisfy the unimodality conditions of \Cref{thm:main}, i.e., they are strictly DGMRD and have finite second moment.}
\label{tab:statics}
\end{table*}

\section{Market Performance}\label{sec:performance}
We now turn to the effects of upstream demand uncertainty on the efficiency of the vertical market. As in \Cref{subn}, we restrict attention to the classic Cournot competition with linear demand and arbitrary number $n$ of competing retailers in the second stage. After scaling $\beta$ to $1$, this implies that the equilibrium order quantities are $q_i^*\(r\)=\frac{1}{n+1}\(\alpha-r\)_+$ for each $i=1,\dots,n$ and any wholesale price $r\ge0$. The supplier's optimal wholesale price, $r^*$, is given by \Cref{thm:main}.

\subsection{Probability of no-trade}
Markets with incomplete information are usually inefficient in the sense that trades which are profitable for all market participants may actually not take place. In the current model, such inefficiencies appear for values of $\alpha$ for which a transaction does not occur in equilibrium under incomplete information although such a transaction would have been beneficial for all parties involved, i.e., supplier, retailers and consumers.\par
If $\alpha<r^*$, then the retailers buy $0$ units and there is an immediate stockout. Hence, for a continuous distribution $F$ of $\alpha$, the probabilitiy of no-trade in equilibrium under incomplete information is equal to $P\(\alpha\le r^*\)=F\(r^*\)$. To study this probability as a measure of market inefficiency, we restrict attention to the family of DMRD distributions, i.e., distributions for which $\m\(r\)$ is non-increasing.
\begin{theorem}\label{thm:notrade}
For any demand distribution $F$ with the DMRD property, the probability $F\(r^*\)$ of no-trade at the equilibrium of the stochastic market cannot exceed the bound $1-e^{-1}$. This bound is tight over all DMRD distributions.
\end{theorem}
\begin{proof}
By expressing the distribution function $F$ in terms of the MRD function, see \cite{Gu88}, we get $F\(r^*\)=1-\frac{\m\(0\)}{\m\(r^*\)}\exp{\left\{-\int_{0}^{r^*}\frac{1}{\m\(u\)}\du\right\}}$. Hence, by the DMRD property and the monotonicity of the exponential function, it follows that $F\(r^*\) \le 1-\frac{\m\(0\)}{\m\(r^*\)}\exp{\left\{-\frac{r^*}{\m\(r^*\)}\right\}}$. Since $r^*=\m\(r^*\)\le \m\(0\)$, we conclude that $F\(r^*\)\le 1- e^{-1}$. If the MRD function is constant, as is the case for the exponential distribution, see \Cref{ex:exponential}, then all inequalities above hold as equalities, which establishes the second claim of the Theorem.
\end{proof} 
\Cref{ex:exponential,ex:beta} highlight the tightness of the no-trade probability bound that is derived in \Cref{thm:notrade}. \Cref{ex:pareto} shows that this bound cannot be extended to the class of DGMRD distributions. The conclusions are summarized in \Cref{prnotrades}.

\begin{example}[Exponential distribution]\label{ex:exponential} Let $\alpha \sim \exp\(\lambda\)$, with $\lambda>0$, and pdf $f\(r\)=\lambda e^{-\lambda r}$ $\mathbf 1_{\{0 \leq r < \infty\}}$. Since $\m\(r\)=1/\lambda$, for $r>0$, the MRD function is constant over its support and, hence, $F$ is both DMRD and IMRD but strictly DGMRD, as $\e\(r\)=1/\lambda r$, for $r>0$. By \Cref{thm:main}, the optimal strategy $r^*$ of the supplier is $r^*=1/\lambda$. The probability of no transaction $F\(r^*\)$ is equal to $F\(r^*\)=F\(1/\lambda\)=1-e^{-1}$, confirming that the bound derived in \Cref{thm:notrade} is tight. Thus, the exponential distribution is the least favorable, over the class of DMRD distributions, in terms of efficiency at equilibrium.
\end{example}

\begin{example}[Beta distribution]\label{ex:beta} This example refers to a special case of the Beta distribution, also known as the Kumaraswamy distribution, see \cite{Jo09}. Let $\alpha \sim Beta\(1,\lambda\)$ with $\lambda>0$, and pdf $f\(r\)= \lambda \(1-r\)^{\lambda-1}\mathbf 1_{\{0<r<1\}}$. Then, $F\(r\) = 1-\(1-r\)^{\lambda}$ and $\m\(r\)=\(1-r\)/\(1+\lambda\)$ for $0 < r < 1$. Since the MRD function is decreasing, \Cref{thm:main} applies and the optimal price of the supplier is $r^*=1/\(\lambda+2\)$. Hence, $F\(r^*\)= 1-\(1-1/\(\lambda+2\)\)^{\lambda} \to 1-e^{-1}$ as $\lambda\to\infty$. This shows that the upper bound of $F\(r^*\)$ in \Cref{thm:notrade} is still tight over distributions with strictly decreasing MRD, i.e., it is not the flatness of the exponential MRD that generated the large inefficiency.
\end{example}

\begin{example}[Generalized Pareto or Pareto II distribution]\label{ex:pareto} This example shows that the bound of \Cref{thm:notrade} does not extend to the class of DGMRD distributions. Let $\alpha \sim \text{GPareto}\(\mu,\sigma,k\)$, with pdf $f\(r\)=\(1+kz\)^{-\(1+1/k\)}/\sigma$ and cdf $F\(r\)=1-\(1+kz\)^{-1/k}$, with $z=\(r-\mu\)/\sigma$. For the parametrization $\mu<\sigma_\epsilon$ and $\sigma_\epsilon=k_\epsilon=\(2+\epsilon\)^{-1}$, with $\epsilon>0$, the cdf becomes $F_\epsilon \(r\)=1-\(1+r-\mu\)^{-\(2+\epsilon\)}$. Moreover, $\ex\alpha_\epsilon^2<\infty$, since $k_\epsilon<1/2$ for any $\epsilon>0$. Hence, by a standard calculation, $\m_\epsilon\(r\)=\(1+r-\mu\)\(1+\epsilon\)$, which shows that $F_\epsilon$ is DGMRD but not DMRD. In this case, $r_\epsilon^*=\(1-\mu\)/\epsilon$ and $F_\epsilon\(r^*\)=1-\(\(1+\epsilon\)\(1-\mu\)/\epsilon\)^{-\(2+\epsilon\)}$, which shows that the probability of a stockout may become arbitrarily large for values of $\epsilon$ close to $0$. The \enquote{pathology} of this example relies on the fact that $\ex\alpha_\epsilon^2\to \infty$ as $\epsilon \searrow 0$.
\end{example}

\begin{figure*}[!bht]
\includegraphics[width=\linewidth]{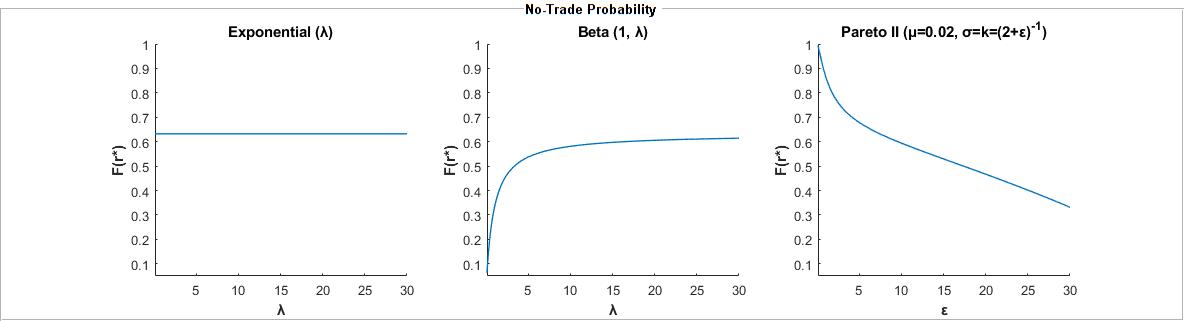}
\caption{Probability $F\(r^*\)$ of no-trade for the Exponential $\(\lambda\)$, Beta $\(1,\lambda\)$ and Pareto II with parameter values $\mu=0.02$, and $\sigma=k=\(2+\epsilon\)^{-1}$ distributions: left, center and right panel respectively. The Exponential and Beta distributions are DMRD and satisfy the $1-e^{-1}$ bound. In contrast, for the choosen range of parameter values, the Pareto II  (or Generalized Pareto) is DGMRD but not DMRD and exhibits no-trade probability that is arbitrarily close to $1$.}
\label{prnotrades}
\end{figure*}

\subsection{Division of realized market profits}
If the realized value of $\alpha$ is larger than $r^*$, then a transaction between the supplier and the retailers takes place. In this case, we measure market efficiency in terms of the \emph{realized market} profits. Specifically, we fix a demand distribution $F$ (which satisfies the sufficiency conditions of \Cref{thm:main}) with support $S$ (with upper and lower bounds $L$ and $H$ respectively, as defined in \Cref{sec:model}) and a realized demand level $\alpha\in S$ and compare the individual realized profits of the supplier and each retailer between the deterministic and the stochastic markets. For clarity, we summarize all related quantities in \Cref{tab:fund}.
\begin{table*}
\centering
\setlength{\tabcolsep}{11pt}
\captionsetup{type=table}
\renewcommand{\arraystretch}{1.4}
\arrayrulecolor{pcolor}
\begin{tabular}{lllll}
\clrtwo
&& \multicolumn{3}{c}{\b Upstream Demand for the Supplier}\\
&&\multicolumn{1}{l}{Uncertain $\alpha\sim F$} &$\phantom{abds}$& \multicolumn{1}{l}{Deterministic $\alpha$} \\
\midrule[0.3mm]
\specialcell{Equilibrium \\[-0.2cm] wholesale price} && \multicolumn{1}{c}{$r^*=\m_F\(r^*\)$} && \multicolumn{1}{c}{$r^*=\alpha/2$}\\[0.1cm]

\clrtwo
&& \multicolumn{3}{c}{\b Realized Profits at Equilibrium}\\
\midrule[0.3mm]
%\clrone
Supplier && $\Pi_s^U=\frac{n}{n+1}r^*\(\alpha-r^*\)_+$ && $\Pi_s^D=\frac{n}{n+1}\(\alpha/2\)^2$\\\\[-0.5cm]
Retailer $i$ && $\Pi_i^U=\frac{1}{\vphantom{\tilde\F}\(n+1\)^2}\(\alpha-r^*\)^2_+$ && $\Pi_i^D=\frac{1}{\vphantom{\tilde\F}\(n+1\)^2}\(\alpha/2\)^2$\\\\[-0.5cm]
%\clrone
Aggregate && $\Pi_\A^U=\Pi_s^U+\sum_{i=1}^n \Pi_i^U$ && $\Pi_\A^D=\Pi_s^U+\sum_{i=1}^n \Pi_i^U$\\\\[-0.5cm]
\bottomrule[0.3mm]
\end{tabular}
\caption{Wholesale price and realized profits in equilibrium for the stochastic (left column) and the deterministic (right column) markets. The realized equilibrium profits correspond to fixed demand level $\alpha\in S$.}
\label{tab:fund}
\end{table*}

We are interested in addressing the following questions: First,  how do the supplier's (retailers') realized profits compare between the stochastic and the deterministic market? Second, how does retail competition and demand uncertainty affect the supplier's (retailers') share of realized market profits? Third, how does the level or retail competition -- number $n$ of retailers -- affect supplier's profits in both markets? The answers are summarized in \Cref{thm:share} which follows rather immediately from \Cref{tab:fund}. To avoid technicalities, we assume throughout that the upper bound $H$ of the support $S$ is large enough, so that $H>2r^*$ (e.g., $H=\infty$).  

\begin{theorem}\label{thm:share}
Let $F$ denote a demand distribution with support $S$ within $L$ and $H$, $r^*$ the respective optimal wholesale price in the stochastic market such that $H>2r^*$, and $\alpha\in S$, with $\alpha>r^*$, a realized demand level, for which trading between supplier and retailers takes place in both the stochastic and the deterministic market. Let, also, $\Pi_s^U/\Pi^U_\A$ and $\Pi_s^D/\Pi^D_\A$ denote the supplier's share of realized profits in the stochastic and deterministic markets respectively. Then,
\begin{enumerate}[label=$\(\roman*\)$,noitemsep]
\item $\Pi_s^U\le \Pi_s^D$, with equality only for $\alpha=2r^*$. In particular, $\Pi_s^U/\Pi_s^D=4\(r^*/\alpha\)\(1-r^*/\alpha\)$ for any $\alpha>r^*$.
\item $\Pi_s^U/\Pi^U_\A$ decreases in the realized demand level $\alpha$.
\item $\Pi_s^D/\Pi^D_\A$ is independent of the demand level $\alpha$.
\item $\Pi_s^U/\Pi^U_\A$ is higher than $\Pi_s^D/\Pi^D_\A$ for values of $\alpha \in \(r^*,2r^*\)$, equal for $\alpha=2r^*$, and lower otherwise.
\item $\Pi_s^D/\Pi^D_\A$ and $\Pi_s^U/\Pi^U_\A$ both increase in the level $n$ of retail competition.
\end{enumerate}
Finally, each retailer's profit in the stochastic market, $\Pi_i^U$, is strictly higher than her profit in the deterministic market $\Pi_i^D$ for all demand levels $\alpha>2r^*$ and less otherwise, with equality for $\alpha=2r^*$ only. 
\end{theorem}

\begin{proof} 
By \Cref{tab:fund}, we have that: (i) $\Pi_s^U\le \Pi_s^D$ if and only if $\frac{n}{n+1}r^*\(\alpha-r^*\)_+\le \frac{n}{n+1}\(\alpha/2\)^2$ which holds with strict inequality for all values of $\alpha$, except for $\alpha=2r^*$ for which the quantities are equal. The second part of statement (i) is immediate. For (ii) $\(\Pi^U_s\)/\(\Pi^U_\A\)=\(nr^*+r^*\)/\(nr^*+\alpha\)$, and for (iii) $\(\Pi^D_s\)/\(\Pi^D_\A\)=\(n+1\)/\(n+2\)$. Now, (iv) and (v) directly follow from the previous calculations. Finally, $\Pi_i^U\ge \Pi_i^D$ if and only if $\frac{1}{\(n+1\)^2}\(\(\alpha-r^*\)_+\)^2\ge \frac{1}{\(n+1\)^2}\(\alpha/2\)^2$ which holds with strict inequality for all values of $\alpha>2r^*$ and with equality for $\alpha=2r^*$.
\end{proof}

The statements of \Cref{thm:share} are rather intuitive and in their largest part, with the exception of part (iv), they conform to earlier findings \citep{Ai12}. (i) The supplier is always better off if he is informed about the retail demand level. (ii) In the stochastic market, he captures a larger share of the realized market profits for lower values of realized demand (but not lower than the no-trade threshold of $r^*$) whereas in the deterministic market (iii) his share of profits is constant with respect to the demand level. (iv) Yet, in the stochastic market, there exists an interval of demand realizations, namely $\(r^*,2r^*\)$, for which the supplier's profits (although less than in the deterministic market) represent a larger share of the aggregate market profits. In any case, (v) retail competition benefits the supplier. Finally, in the case that the supplier prices under uncertainty, each retailer makes a larger profit for higher realized demand values which abides to intuition. These observations conform with the existence of conflicting incentives regarding demand-information disclosure between the retailers and the supplier, cf. \cite{Li17}.

%\begin{wrapfigure}[13]{R}{0.49\textwidth}
%\vspace*{-0.4cm}\includegraphics[width=\linewidth]{supratio4}
%\caption{\label{supratio}Ratio $\Pi_s^U/\Pi_s^D$ of the supplier's realized profits with and without demand uncertainty for $\alpha\sim$ Weibull $\(1,2\)$ and arbitrary $n$.}
%\end{wrapfigure}

\subsubsection{Deterministic and stochastic markets: aggregate profits}\label{sub:aggregate}
We next turn to the comparison of the aggregate market profits between the deterministic and the stochastic market. As before, we fix a demand distribution $F$ (which is again assumed to satisfy the sufficiency conditions of \Cref{thm:main}) with support $S$ within $L$ and $H$, and evaluate the ratio $\Pi_\A^U/\Pi_\A^D$ of the aggregate realized market profits in the stochastic market to the aggregate market profits in the deterministic market. To study market performance under the two scenarios, we need to evaluate the combined effect of demand uncertainty and retail competition. For a realized demand $\alpha\le r^*$, there is a stockout and the realized aggregated profits $\Pi_\A^U$ are equal to $0$. In this case, the stochastic market performs arbitrarily worse than the deterministic market and the ratio is equal to $0$ for any number $n\ge 1$ of competing retailers. Hence, for a non-trivial analysis, we restrict attention to $\alpha>r^*$ for which trading takes place in both the stochastic and the deterministic markets.
\begin{theorem}\label{thm:agg}
Let $F$ denote a demand distribution with support $S$ within $L$ and $H$, with $H$ large enough, and let $r^*$ denote the respective optimal wholesale price in the stochastic market. Additionally, suppose that $\alpha\in S$, with $\alpha>r^*$ is a realized demand level for which trading between supplier and retailers takes place in both the stochastic and the deterministic market. Let, also, $\Pi_\A^U/\Pi^D_\A$ denote the ratio of the aggregate realized profits in the stochastic market to the aggregate profits in the deterministic market. Then,
\begin{enumerate}[label=$\(\roman*\)$,noitemsep]
\item $\Pi_\A^U/\Pi_\A^D> 1$ for $\alpha>2r^*$ if $n=1$ or $n=2$ and for $\alpha\in \(2r^*, \frac{2n}{n-2}r^*\)$ if $n\ge 3$. 
\item $\Pi_\A^U/\Pi_\A^D$ is maximized for $\alpha^*=2nr^*/\(n-1\)$ for $n\ge 2$, for which it is equal to $1+\(n\(n+2\)\)^{-1}$. Moreover, $\Pi_\A^U/\Pi_\A^D$ converges to $4/\(n+2\)$ as $\alpha \to\infty$ for any $n\ge1$.
\item $\Pi_\A^U/\Pi_\A^D$ increases in the level of competition for demand levels $\alpha<2r^*$ and decreases thereafter.
\end{enumerate}
\end{theorem}
Again, to avoid unnecessary technicalities in the proof of \Cref{thm:agg}, we assume that $H$ is large enough, e.g., $H=\infty$. 
\begin{proof} By \Cref{tab:fund}, a direct substitution yields that $\Pi_\A^U-\Pi_\A^D>0$ iff 
\[\frac{2-n}4\cdot \alpha^2+r^*\(n-1\)\alpha -n\(r^*\)^2>0\] with $\alpha>r^*$. For $n=1,2$, the result is straightforward, whereas for $n\ge3$ the result follows from the observation that the roots of the expression in the left part are given by $\alpha_{1,2}=2r^*\cdot\frac{n-1\pm1}{n-2}$. This establishes (i) and after some trivial algebra, also (ii). To obtain (iii), we compare $\Pi_\A^U/\Pi_\A^D$ for arbitrary $n$ to $\Pi_\A^U/\Pi_\A^D$ for $n+1$
\[\frac{4\(\alpha-r^*\)\(\alpha+\(n+1\)r^*\)}{\alpha^2\(n+3\)}>\frac{4\(\alpha-r^*\)\(\alpha+nr^*\)}{\alpha^2\(n+2\)}\iff 2r^*>\alpha\]
which yields the statement. 
\end{proof}
Statement (i) of \Cref{thm:agg} asserts that there exists an interval of realized demand values, whose upper bound depends on the number $n$ of competing retailers, for which the stochastic market outperforms the deterministic market in terms of aggregate profits. The effect of increasing retail competition on the aggregate profits of the stochastic market is twofold. First, the range (interval) of demand values for which the ratio of aggregate profits exceeds $1$ reduces to a single point as competition increases ($n\to\infty$). Second, for larger values of realized demand, the ratio converges to $4/\(n+2\)$ as $\alpha\to\infty$. This shows that uncertainty on the side of the supplier is less detrimental for the aggregate market profits when the level of retail competition is low. In particular, for $n=1,2$, the aggregate profits of the stochastic market remain strictly higher than the profits of the deterministic market for all large enough realized demand levels. As competition increases this remains true only for lower (but still above the no-trade threshold) demand levels. However, for higher demand realizations, the ratio degrades linearly in the number of competing retailers. \par
The statements of \Cref{thm:agg} are illustrated in \Cref{ncompare}. Here $\alpha\sim \text{Gamma}\(2,2\)$ but the picture is essentially the same for any choice of demand distribution that satisfies the sufficiency conditions of \Cref{thm:main} and for which $H$ is large enough, i.e., $H>2r^*$.

\begin{figure*}[!htp]
\includegraphics[width=\linewidth]{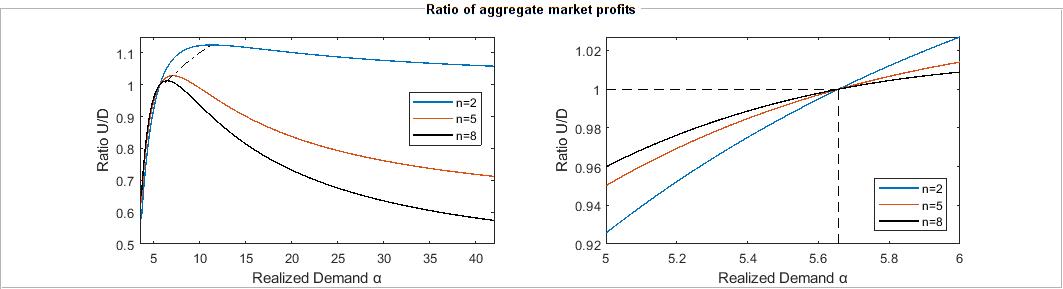}
\caption{The left panel depicts the ratio $\Pi_\A^U/\Pi_\A^D$ for $n=2, 5$ and $8$, where $\alpha \sim \text{Gamma}\(2,2\)$. The dashed line shows the points $\alpha=2nr^*/\( n-1\)$ in which the ratios are maximized, taking the value $1+\(n\(n+2\)\)^{-1}$. The right panel magnifies the interval $[4,7]$ around the intersection point $\alpha=2r^*$ of the three curves. The ratio  increases in $n$ prior to the intersection point, $2r^*\approx 5.657$, and decreases in $n$ thereafter.}
\label{ncompare}
\end{figure*}

\section{Conclusions}\label{sec:conclusions}
In this paper, we revisited the classic problem of optimal pricing by a monopolist who is facing linear stochastic demand. The monopolist may sell directly to the consumers or via a retail market with an arbitrary competition structure. Our main theoretical finding is that the price elasticity of expected demand, and hence also the monopolist's optimal prices, can be expressed in terms of the mean residual demand (MRD) function of the demand distribution. In economic terms, the MRD function describes the expected additional demand given that current demand has reached or exceeded a certain threshold. This leads to a closed form characterization of the points of unitary elasticity that maximize the monopolist's profits and the derivation of a mild unimodality condition for the monopolist's objective function that generalizes the widely used \emph{increasing generalized failure rate} (IGFR) condition. A direct byproduct is a distribution free and tight bound on the probability of no trade between the supplier and the retailers. \par
When we compare optimal prices between markets with different demand characteristics, the main implication of the above characterization is that it allows us to exploit various forms of knowledge on the demand distribution via the theory of stochastic orderings. Specifically, if two markets can be ordered in terms of their mean residual demand function, then the seller's optimal prices can be ordered accordingly. This establishes a link between the price elasticity of expected demand, which is naturally the critical determinant of price movements in response to changes in demand, and some well-understood characteristics of the demand distribution. The stochastic orders approach works under various informational assumptions on the demand distribution and provides a way to systematically exploit the abundance of data that firms possess about historical price/demand. \par 
From a managerial perspective, our study provides a tractable theoretical framework which can be used to reason about price changes in advance of anticipated market movements or to benchmark data-driven predictions that are derived from pricing analytics. Our results suggest that the effects of market size and demand variability on prices critically depend on the notions of size and variability that will be employed. This implies that exact predictions about price movements can only be done in a case by case basis and should only be used with caution. Such tools are particularly useful in industries in which fixing a price often precedes the demand realization such as subscription based businesses or businesses that sell durable goods, tickets or leisure time services. More generally, our findings can be used to explain the diversity of price responses to market characteristics that are observed in practice and provide a diverse toolbox for managers to optimally set and adjust prices under different market conditions.

\section*{Acknowledgements}
Stefanos Leonardos gratefully acknowledges support by the Alexander S. Onassis Public Benefit Foundation and partial support by NRF 2018 Fellowship NRF-NRFF2018-07.

\Urlmuskip=0mu plus 1mu
\bibliographystyle{plainnat}
\bibliography{staticsbib_R1}
\end{document}